\documentclass[12pt]{article}

\font\smallit=cmti10

\usepackage{amssymb,
amsmath,
amsthm} 
\usepackage{xcolor}
\usepackage{hyperref}

\hypersetup{
linkcolor=magenta,
urlcolor=magenta,
citecolor=blue,
colorlinks=true}

\makeatletter

\renewcommand\section{\@startsection {section}{1}{\z@}
{-30pt \@plus -1ex \@minus -.2ex}
{2.3ex \@plus.2ex}
{\normalfont\normalsize\bfseries}}

\renewcommand\subsection{\@startsection{subsection}{2}{\z@}
{-3.25ex\@plus -1ex \@minus -.2ex}
{1.5ex \@plus .2ex}
{\normalfont\normalsize\bfseries}}

\renewcommand{\@seccntformat}[1]{\csname the#1\endcsname. }

\makeatother

\newtheorem{theorem}{Theorem}
\newtheorem{lemma}[theorem]{Lemma}

\newtheorem{proposition}[theorem]{Proposition}

\newtheorem{remark}[theorem]{Remark}
\newtheorem{example}[theorem]{Example}

\DeclareMathOperator{\SL2Z}{\text{\rm SL}(2,\mathbb{Z})}

\begin{document}

\begin{center}
\uppercase{\bf A Note on Boolean Lattices and Farey Sequences III}
\vskip 20pt
{\bf Andrey\,O.\,Matveev
}
\\
{\smallit $\quad$}\\
{\tt andrey.o.matveev@gmail.com}
\end{center}
\vskip 30pt
$\quad$
\vskip 30pt

\centerline{\bf Abstract}
\noindent
We describe monotone maps between subsequences of the~Farey sequences.

\pagestyle{myheadings}
\thispagestyle{empty}
\baselineskip=12.875pt
\vskip 30pt

\section{Introduction}

Let $\mathbb{V}(n)$ be an $n$--dimensional
linear space,
and $\mathbf{A}$ its proper $m$--dimensional
subspace. We associate with the integers $n$ and $m$ the increasing sequence of irreducible fractions
\begin{multline}
\label{eq:63}
\mathcal{F}\bigl(\mathbb{B}(n),m\bigr):=\Bigl(\tfrac{\dim(\mathbf{B}\cap\mathbf{A})}{\gcd(\dim(\mathbf{B}\cap\mathbf{A}),
\dim\mathbf{B})}\!\Bigm/\!\tfrac{\dim\mathbf{B}}{\gcd(\dim(\mathbf{B}\cap\mathbf{A}),
\dim\mathbf{B})}:\\
\mathbf{B}\text{\
subspace of\ } \mathbb{V}(n),\ \dim\mathbf{B}>0\Bigr)\; ;
\end{multline}
in other words,
\begin{equation}
\label{eq:64}
\mathcal{F}\bigl(\mathbb{B}(n),m\bigr):=\left(\tfrac{h}{k}\in\mathcal{F}_n:\ m+k-n\leq h\leq m\right)\; ,
\end{equation}
where $\mathcal{F}_n$ denotes the Farey sequence of order $n$. See e.g.~\cite{Aigner,Anglin,Burton,Hata,Hatcher,Hua,Huxley-A,Huxley-D,Karpenkov,Khrushchev,Oswald-Steuding,Rademacher,Schroeder} on the Farey sequences. In particular,
\begin{equation}
\label{eq:68}
\mathcal{F}\bigl(\mathbb{B}(2m),m\bigr):=\left(\tfrac{h}{k}\in\mathcal{F}_{2m}:\ k-m\leq h\leq m\right)\; .
\end{equation}
Equivalent descriptions of Farey subsequence~(\ref{eq:64})
can be given via the cardinalities of subsets of an $n$--set, or via the ranks of elements of the Boolean lattice $\mathbb{B}(n)$ of rank~$n$, see~\cite{M-Integers-I,M-Integers-II}.

Sequence~(\ref{eq:64}) can be regarded as the intersection
\begin{equation}
\label{eq:66}
\mathcal{F}_n^m\cap\mathcal{G}_n^m=:\mathcal{F}\bigl(\mathbb{B}(n),m\bigr)
\end{equation}
of the Farey subsequence
\begin{align}
\label{eq:2}
\mathcal{F}_n^m:&=\left(\tfrac{h}{k}\in\mathcal{F}_n:\ h\leq m\right)\\
\intertext{introduced in~\cite{A-Z}, and of the Farey subsequence}
\label{eq:3}
\mathcal{G}_n^m:&=\left(\tfrac{h}{k}\in\mathcal{F}_n:\ m+k-n\leq h\right)\; ;
\end{align}
it follows from definitions~(\ref{eq:2}) and~(\ref{eq:3}) that
\begin{equation}
\label{eq:65}
\mathcal{F}_n^m\cup\mathcal{G}_n^m=\mathcal{F}_n\; .
\end{equation}

Recall\hfill that\hfill the\hfill maps\hfill $\mathcal{F}_n\to\mathcal{F}_n$,\hfill $\mathcal{F}_n^m\to\mathcal{G}_n^{n-m}$,\hfill $\mathcal{G}_n^m\to\mathcal{F}_n^{n-m}$,\hfill
and\hfill $\mathcal{F}(\mathbb{B}(n),m)$\newline \mbox{$\to\mathcal{F}(\mathbb{B}(n),n-m)$}, such that
\begin{equation}
\label{eq:1}
\tfrac{h}{k}\mapsto\tfrac{k-h}{k}\; ,\ \ \ \phantom{\mapsto} \left[\begin{smallmatrix}\!h\!\\
\!k\!\end{smallmatrix}\right]\mapsto\left[\begin{smallmatrix}-1&1\\0&1\end{smallmatrix}\right]\cdot
\left[\begin{smallmatrix}\!h\!\\ \!k\!\end{smallmatrix}\right]\; ,
\end{equation}
where $\left[\begin{smallmatrix}\!h\!\\ \!k\!\end{smallmatrix}\right]$ is a vector presentation of $\tfrac{h}{k}$, are all order--reversing and bijective.

In\hfill analogy\hfill to\hfill sequences\hfill (\ref{eq:2})\hfill and\hfill (\ref{eq:3}),\hfill we\hfill define\hfill similar\hfill subsequences\hfill of\hfill the\hfill se-\newline quence~$\mathcal{F}\bigl(\mathbb{B}(n),m\bigr)$ as follows: given an integer $\ell$, $1\leq\ell\leq m$,
\begin{align}
\label{eq:69}
\mathcal{F}\bigl(\mathbb{B}(n),m\bigr)^{\ell}:&=
\left(\tfrac{h}{k}\in\mathcal{F}\bigl(\mathbb{B}(n),m\bigr):\ h\leq \ell\right)\; ;
\\
\intertext{given an integer $\ell$, $m\leq\ell\leq n-1$,}
\label{eq:70}
\mathcal{G}\bigl(\mathbb{B}(n),m\bigr)^{\ell}:&=
\left(\tfrac{h}{k}\in\mathcal{F}\bigl(\mathbb{B}(n),m\bigr):\ \ell+k-n\leq h\right)\; ;
\end{align}
thus, $\mathcal{F}\bigl(\mathbb{B}(n),m\bigr)^m=\mathcal{G}\bigl(\mathbb{B}(n),m\bigr)^m=\mathcal{F}\bigl(\mathbb{B}(n),m\bigr)$.
It is easy to see, as noted in~Section~\ref{sec:1}, that the sequences $\mathcal{F}\bigl(\mathbb{B}(n),m\bigr)^{\ell}$ and
$\mathcal{G}\bigl(\mathbb{B}(n),m\bigr)^{\ell}$ are both sequences of the form
$\mathcal{F}\bigl(\mathbb{B}(\boldsymbol{\cdot}),\boldsymbol{\cdot}\bigr)$ for any allowed~$\ell$\;\!'s; moreover, if $\lambda$ is an integer such that~$1\leq\lambda\leq n-m$, then
\begin{equation}
\mathcal{F}\bigl(\mathbb{B}(n),m\bigr)^{\ell}\cap\mathcal{G}\bigl(\mathbb{B}(n),m\bigr)^{n-\lambda}=
\mathcal{F}\bigl(\mathbb{B}(\ell+\lambda),\ell\bigr)\; .
\end{equation}
In~Section~\ref{section2s}, we describe monotone  maps between (sub)sequences of the~Farey sequences of orders $2^s m$, $2^{s+1}m$ and $2^{s+2}m$. In~Section~\ref{comp-maps}, these observations are reformulated for the~Farey sequences of arbitrary orders. Supplementary information is collected in~Appendix~\ref{sec-app}.

\section{The Farey Subsequences $\mathcal{F}\bigl(\mathbb{B}(n),m\bigr)^{\ell}$ and $\mathcal{G}\bigl(\mathbb{B}(n),m\bigr)^{\ell}$}

\label{sec:1}

It\hfill is\hfill easily\hfill verified\hfill that\hfill the\hfill subsequences\hfill $\mathcal{F}\bigl(\mathbb{B}(n),m\bigr)^{\ell}$\hfill and\hfill $\mathcal{G}\bigl(\mathbb{B}(n),m\bigr)^{\ell}$\hfill of\hfill the\hfill se-\newline quence~$\mathcal{F}\bigl(\mathbb{B}(n),m\bigr)$, defined by Eqs.~(\ref{eq:69}) and~(\ref{eq:70}) respectively, are also sequences of the form~(\ref{eq:64}); indeed, they can be redefined as follows:
\begin{align}
\mathcal{F}\bigl(\mathbb{B}(n),m\bigr)^{\ell}:&=\mathcal{F}_{n-m+\ell}^{\ell}\cap\mathcal{G}_{n-m+\ell}^{\ell}\; , &
\mathcal{G}\bigl(\mathbb{B}(n),m\bigr)^{\ell}:&=\mathcal{F}_{n+m-\ell}^m\cap\mathcal{G}_{n+m-\ell}^m\; ;\\
\intertext{in particular,}
\mathcal{F}\bigl(\mathbb{B}(2m),m\bigr)^{\ell}:&=\mathcal{F}_{m+\ell}^{\ell}\cap\mathcal{G}_{m+\ell}^{\ell}\; , &
\mathcal{G}\bigl(\mathbb{B}(2m),m\bigr)^{\ell}:&=\mathcal{F}_{3m-\ell}^m\cap\mathcal{G}_{3m-\ell}^m\; .
\end{align}

\begin{remark} {\rm(i)} If $\ell\in\mathbb{P}$, $1\leq\ell\leq m$, then
\label{th:1}
\begin{align}
\label{eq:71}
\mathcal{F}\bigl(\mathbb{B}(n),m\bigr)^{\ell}&=\mathcal{F}\bigl(\mathbb{B}(n-m+\ell),\ell\bigr)\; , &
\mathcal{F}\bigl(\mathbb{B}(2m),m\bigr)^{\ell}&=\mathcal{F}\bigl(\mathbb{B}(m+\ell),\ell\bigr)\; .\\
\intertext{\noindent{\rm(ii)} If $\ell\in\mathbb{P}$, $m\leq\ell\leq n-1$, then}
\label{eq:72}
\mathcal{G}\bigl(\mathbb{B}(n),m\bigr)^{\ell}&=\mathcal{F}\bigl(\mathbb{B}(n+m-\ell),m\bigr)\; , &
\mathcal{G}\bigl(\mathbb{B}(2m),m\bigr)^{\ell}&=\mathcal{F}\bigl(\mathbb{B}(3m-\ell),m\bigr)\; .
\end{align}
\end{remark}

\begin{proposition} The maps
\begin{align}
\mathcal{F}\bigl(\mathbb{B}(n),m\bigr)^{\ell}&\to\mathcal{F}\bigl(\mathbb{B}(n),
n-\ell\bigr)^{n-m}\; ,\\
\mathcal{F}\bigl(\mathbb{B}(2m),m\bigr)^{\ell}&\to\mathcal{F}\bigl(\mathbb{B}(2m),
2m-\ell\bigr)^m\; ,\\
\mathcal{G}\bigl(\mathbb{B}(n),m\bigr)^{\ell}&\to\mathcal{G}\bigl(\mathbb{B}(n),
n-\ell\bigr)^{n-m}\; ,\\
\mathcal{G}\bigl(\mathbb{B}(2m),m\bigr)^{\ell}&\to\mathcal{G}\bigl(\mathbb{B}(2m),
2m-\ell\bigr)^m\; ,
\end{align}
defined by Eq.~{\rm(\ref{eq:1})}, are order--reversing and bijective.
\end{proposition}

\begin{proof}
Note that $\mathcal{F}\bigl(\mathbb{B}(n),n-\ell\bigr)^{n-m}=\mathcal{F}\bigl(\mathbb{B}(n-m+\ell),n-m\bigr)$, by Remark~\ref{th:1}(i), and apply the
order--reversing bijection
from Eq.~(\ref{eq:1}) to the sequences~\mbox{$\mathcal{F}\bigl(\mathbb{B}(n-m+\ell),\ell\bigr)$} $=\mathcal{F}\bigl(\mathbb{B}(n),m\bigr)^{\ell}$ and
$\mathcal{F}\bigl(\mathbb{B}(n-m+\ell),n-m\bigr)$.

We have $\mathcal{G}\bigl(\mathbb{B}(n),n-\ell\bigr)^{n-m}=\mathcal{F}\bigl(\mathbb{B}(n+m-\ell),n-\ell\bigr)$, by Remark~\ref{th:1}(ii), and Eq.~(\ref{eq:1}) provides the order--reversing bijection between the sequences~\mbox{$\mathcal{F}\bigl(\mathbb{B}(n+m-\ell),m\bigr)$} $=\mathcal{G}\bigl(\mathbb{B}(n),m\bigr)^{\ell}$ and~$\mathcal{F}\bigl(\mathbb{B}(n+m-\ell),n-\ell\bigr)$.
\end{proof}

\begin{example}

Suppose $n:=6$, $m:=4$, $\ell':=3$ and $\ell'':=5$.

{
\footnotesize
\begin{align*}
\mathcal{F}_n &= \bigl(\tfrac{0}{1}<\tfrac{1}{6}<\tfrac{1}{5}<\tfrac{1}{4}<\tfrac{1}{3}<\tfrac{2}{5}<
\tfrac{1}{2}<\tfrac{3}{5}<\tfrac{2}{3}<\tfrac{3}{4}<\tfrac{4}{5}<\tfrac{5}{6}<\tfrac{1}{1}\bigr)\; ,\\
\mathcal{F}^m_n &= \bigl(\tfrac{0}{1}<\tfrac{1}{6}<\tfrac{1}{5}<\tfrac{1}{4}<\tfrac{1}{3}<\tfrac{2}{5}<
\tfrac{1}{2}<\tfrac{3}{5}<\tfrac{2}{3}<\tfrac{3}{4}<\tfrac{4}{5}<\phantom{\tfrac{5}{6}}<\tfrac{1}{1}\bigr)\; ,\\
\mathcal{G}^m_n &= \bigl(\tfrac{0}{1}<\phantom{\tfrac{1}{6}}<\phantom{\tfrac{1}{5}}<\phantom{\tfrac{1}{4}}
<\tfrac{1}{3}<\phantom{\tfrac{2}{5}}<
\tfrac{1}{2}<\tfrac{3}{5}<\tfrac{2}{3}<\tfrac{3}{4}<\tfrac{4}{5}<\tfrac{5}{6}<\tfrac{1}{1}\bigr)\; ,\\
\mathcal{F}\bigl(\mathbb{B}(n),m\bigr) &= \bigl(\tfrac{0}{1}<\phantom{\tfrac{1}{6}}<\phantom{\tfrac{1}{5}}<\phantom{\tfrac{1}{4}}
<\tfrac{1}{3}<\phantom{\tfrac{2}{5}}<
\tfrac{1}{2}<\tfrac{3}{5}<\tfrac{2}{3}<\tfrac{3}{4}<\tfrac{4}{5}<\phantom{\tfrac{5}{6}}<\tfrac{1}{1}\bigr)\; ,\\
\mathcal{F}\bigl(\mathbb{B}(n),m\bigr)^{\ell'}\\=\mathcal{F}\bigl(\mathbb{B}(n-m+\ell'),\ell'\bigr)&= \bigl(\tfrac{0}{1}<\phantom{\tfrac{1}{6}}<\phantom{\tfrac{1}{5}}<\phantom{\tfrac{1}{4}}
<\tfrac{1}{3}<\phantom{\tfrac{2}{5}}<
\tfrac{1}{2}<\tfrac{3}{5}<\tfrac{2}{3}<\tfrac{3}{4}<\phantom{\tfrac{4}{5}}<\phantom{\tfrac{5}{6}}<\tfrac{1}{1}\bigr)\; ,\\
\mathcal{F}\bigl(\mathbb{B}(n),n-\ell'\bigr)^{n-m}\\=\mathcal{F}\bigl(\mathbb{B}(n-m+\ell'),n-m\bigr)&= \bigl(\tfrac{0}{1}<\phantom{\tfrac{1}{6}}<\phantom{\tfrac{1}{5}}<\tfrac{1}{4}<\tfrac{1}{3}<\tfrac{2}{5}<
\tfrac{1}{2}<\phantom{\tfrac{3}{5}}<\tfrac{2}{3}<\phantom{\tfrac{3}{4}}<\phantom{\tfrac{4}{5}}<\phantom{\tfrac{5}{6}}<\tfrac{1}{1}\bigr)\; ,\\
\mathcal{G}\bigl(\mathbb{B}(n),m\bigr)^{\ell''}\\=\mathcal{F}\bigl(\mathbb{B}(n+m-\ell''),m\bigr)&= \bigl(\tfrac{0}{1}<\phantom{\tfrac{1}{6}}<\phantom{\tfrac{1}{5}}<\phantom{\tfrac{1}{4}}<\phantom{\tfrac{1}{3}}<\phantom{\tfrac{2}{5}}<
\tfrac{1}{2}<\phantom{\tfrac{3}{5}}<\tfrac{2}{3}<\tfrac{3}{4}<\tfrac{4}{5}<\phantom{\tfrac{5}{6}}<\tfrac{1}{1}\bigr)\; ,\\
\mathcal{G}\bigl(\mathbb{B}(n),n-\ell''\bigr)^{n-m}\\=\mathcal{F}\bigl(\mathbb{B}(n+m-\ell''),n-\ell''\bigr)&= \bigl(\tfrac{0}{1}<\phantom{\tfrac{1}{6}}<\tfrac{1}{5}<\tfrac{1}{4}<\tfrac{1}{3}<\phantom{\tfrac{2}{5}}<
\tfrac{1}{2}<\phantom{\tfrac{3}{5}}<\phantom{\tfrac{2}{3}}<\phantom{\tfrac{3}{4}}<\phantom{\tfrac{4}{5}}<\phantom{\tfrac{5}{6}}<\tfrac{1}{1}\bigr)\; .
\end{align*}
}
\end{example}

\section{The Farey (Sub)sequences $\mathcal{F}_{2^s m}$, $\mathcal{F}\bigl(\mathbb{B}(2^{s+1}m),2^s m\bigr)$,
$\mathcal{F}\bigl(\mathbb{B}(2^{s+2}m),2^{s+1}m\bigr)$, and~Mo\-no\-tone Maps}

\label{section2s}

In this section, we consider
monotone maps between (sub)sequences of the sequences~$\mathcal{F}_{2^s m}$, $\mathcal{F}\bigl(\mathbb{B}(2^{s+1}m),2^s m\bigr)$
and $\mathcal{F}\bigl(\mathbb{B}(2^{s+2}m),2^{s+1}m\bigr)$, where $m\in\mathbb{P}$, $s\in\mathbb{N}$. The observations made for these maps can be used as an instructive induction step in an analysis of composite maps between Farey (sub)sequences that are described in Section~\ref{comp-maps}.

Recall that for the sequences $\mathcal{F}\bigl(\mathbb{B}(n),m\bigr)$ and, in particular, for the sequences $\mathcal{F}\bigl(\mathbb{B}(2m),m\bigr)$, it is convenient to treat the subsequences
\begin{align}
\mathcal{F}^{\leq\frac{1}{2}}\bigl(\mathbb{B}(n),m\bigr):&=\bigl(\tfrac{h}{k}\in\mathcal{F}\bigl(\mathbb{B}(n),m\bigr):\ \tfrac{h}{k}\leq\tfrac{1}{2}\bigr)\\
\intertext{and}
\mathcal{F}^{\geq\frac{1}{2}}\bigl(\mathbb{B}(n),m\bigr):&=\bigl(\tfrac{h}{k}\in\mathcal{F}\bigl(\mathbb{B}(n),m\bigr):\ \tfrac{h}{k}\geq\tfrac{1}{2}\bigr)
\end{align}
separately~\cite[Thm.~5, Lem.~3, Cor.~4]{M-Integers-I}:

\noindent$\bullet$ The maps
\begin{align}
\label{eq:4}
\mathcal{F}_{m}
&\to\mathcal{F}^{\leq\frac{1}{2}}\bigl(\mathbb{B}(2m),m\bigr)\; , &
\left[
\right]\; ,
\intertext{are order--preserving and bijective. The maps}
\label{eq:16}
\mathcal{F}_{m}&\to\mathcal{F}^{\leq\frac{1}{2}}\bigl(\mathbb{B}(2m),m\bigr)\; , &
\left[%
\right]\; ,
\end{align}
are order--reversing and bijective.

\noindent$\bullet$ The maps
\begin{align}
\label{eq:6}
\mathcal{F}^{\leq\frac{1}{2}}\bigl(\mathbb{B}(2m),m\bigr)&\to\mathcal{F}^{\leq\frac{1}{2}}\bigl(\mathbb{B}(2m),m\bigr)\; , &
\left[%
\right]\; ,
\end{align}
are order--reversing and bijective.

\noindent$\bullet$ The maps
\begin{align}
\label{eq:8}
\mathcal{F}^{\leq\frac{1}{2}}\bigl(\mathbb{B}(2m),m\bigr)&\to\mathcal{F}^{\geq\frac{1}{2}}\bigl(\mathbb{B}(2m),m\bigr)\; , &
\left[%
\right]\; ,
\end{align}
are order--preserving and bijective.

\begin{remark}
\label{th:4}
Let $m\in\mathbb{P}$, $s\in\mathbb{N}$. {\rm(i)}
Eqs.~{\rm(\ref{eq:4})} and~{\rm(\ref{eq:15})} imply that the following maps are order--preserving and injective:
{\small
\begin{align}
\label{eq:74}
\mathcal{F}_{2^s m}
&\to\mathcal{F}^{\leq\frac{1}{2}}\bigl(\mathbb{B}(2^{s+2}m),2^{s+1}m\bigr)\; , &
\left[%
\right]\; .
\end{align}
}

{\rm(ii)} As a consequence, the following maps are order--reversing and injective:
{\small
\begin{align}
\label{eq:78}
\mathcal{F}_{2^s m}
&\to\mathcal{F}^{\leq\frac{1}{2}}\bigl(\mathbb{B}(2^{s+2}m),2^{s+1}m\bigr)\; , &
\left[%
\right]\; .
\end{align}
}
\end{remark}

\begin{lemma} Let $m\in\mathbb{P}$, $s\in\mathbb{N}$. {\rm(i)} The following maps are order--preserving and injective:
{\small
\begin{align}
\label{eq:17}
\mathcal{F}^{\leq\frac{1}{2}}\bigl(\mathbb{B}(2^{s+1}m),2^s m\bigr)&\to\mathcal{F}^{\leq\frac{1}{2}}\bigl(\mathbb{B}(2^{s+2}m),2^{s+1}m\bigr)\; , &
\left[%
\right]\; .
\end{align}
}

{\rm(ii)} The following maps are order--reversing and injective:
{\small
\begin{align}
\label{eq:27}
\mathcal{F}^{\leq\frac{1}{2}}\bigl(\mathbb{B}(2^{s+1}m),2^s m\bigr)&\to\mathcal{F}^{\leq\frac{1}{2}}\bigl(\mathbb{B}(2^{s+2}m),2^{s+1}m\bigr)\; , &
\left[%
\right]\; .
\end{align}
}
\end{lemma}

\begin{proof} (i): (\ref{eq:17}) $\Leftarrow$ (\ref{eq:4}); (\ref{eq:18}) $\Leftarrow$ (\ref{eq:9}) \& (\ref{eq:4});
(\ref{eq:21}) $\Leftarrow$ (\ref{eq:8}) \& (\ref{eq:4}); (\ref{eq:22}) $\Leftarrow$ (\ref{eq:4}); (\ref{eq:23})~$\Leftarrow$~(\ref{eq:15});
(\ref{eq:24}) $\Leftarrow$ (\ref{eq:9}) \& (\ref{eq:15}); (\ref{eq:25}) $\Leftarrow$ (\ref{eq:8}) \& (\ref{eq:15});
(\ref{eq:26}) $\Leftarrow$ (\ref{eq:15}).

(ii): (\ref{eq:27}) $\Leftarrow$ (\ref{eq:6}) \& (\ref{eq:17}); (\ref{eq:28}) $\Leftarrow$ (\ref{eq:7}) \& (\ref{eq:28});
(\ref{eq:29}) $\Leftarrow$ (\ref{eq:6}) \& (\ref{eq:21});  (\ref{eq:30})~$\Leftarrow$~(\ref{eq:7})~\&~(\ref{eq:22});
(\ref{eq:31}) $\Leftarrow$ (\ref{eq:6}) \& (\ref{eq:23}); (\ref{eq:32}) $\Leftarrow$ (\ref{eq:7}) \& (\ref{eq:24});
(\ref{eq:33}) $\Leftarrow$ (\ref{eq:6}) \& (\ref{eq:25}); (\ref{eq:34}) $\Leftarrow$ (\ref{eq:7}) \& (\ref{eq:26}).
\end{proof}

Let $\mathcal{C}$ be an increasing sequence of irreducible fractions $\tfrac{h}{k}$ written in vector form $\left[
\right]$. If~\mbox{$\boldsymbol{S}\in\SL2Z$}, then we denote by~$\boldsymbol{S}\cdot\mathcal{C}$ the sequence of  vectors~\mbox{$\boldsymbol{S}\cdot\left[%
\right]\in\mathcal{C}\right)\; .
\end{equation}

\begin{proposition}
\label{th:5}
Let $m\in\mathbb{P}$, $s\in\mathbb{N}$.
{\rm(i)} The  following  maps  between subsequences of the sequence~$\mathcal{F}\bigl(\mathbb{B}(2^{s+2}m),2^{s+1}m\bigr)$
are order--preserving and bijective:
{\small
\begin{align}
\label{eq:35}
\left[%
\right]\; .
\end{align}
}

{\rm(ii)(a)} The following maps between subsequences of the sequence $\mathcal{F}\bigl(\mathbb{B}(2^{s+2}m),2^{s+1}m\bigr)$
are order--reversing and bijective:
{\small
\begin{align}
\label{eq:47}
\left[%
\right]\; .
\end{align}
}

{\rm(ii)(b)} The following maps of subsequences of the sequence $\mathcal{F}\bigl(\mathbb{B}(2^{s+2}m),2^{s+1}m\bigr)$ onto themselves
are order--reversing and bijective:
{\small
\begin{align}
\label{eq:59}
\left[%
\right]\; .
\end{align*}
}
\end{proof}

\begin{example}
\label{th:6}
Suppose $m:=1$, $s:=1$, and consider the sequences
{\scriptsize
\begin{align*}
\mathcal{F}_{2^{s+2}m} &= \bigl(\tfrac{0}{1}<\tfrac{1}{8}<\tfrac{1}{7}<\tfrac{1}{6}<\tfrac{1}{5}<\tfrac{1}{4}<\tfrac{2}{7}
<\tfrac{1}{3}<\tfrac{3}{8}<\tfrac{2}{5}<\tfrac{3}{7}<
\tfrac{1}{2}<\tfrac{4}{7}<\tfrac{3}{5}<\tfrac{5}{8}<\tfrac{2}{3}<\tfrac{5}{7}<\tfrac{3}{4}<\tfrac{4}{5}<\tfrac{5}{6}<\tfrac{6}{7}<
\tfrac{7}{8}<\tfrac{1}{1}\bigr)\; ,\\
\mathcal{F}\bigl(\mathbb{B}(2^{s+2}m),2^{s+1}m\bigr) &= \bigl(\tfrac{0}{1}<\phantom{\tfrac{1}{8}}<\phantom{\tfrac{1}{7}}<\phantom{\tfrac{1}{6}}<\tfrac{1}{5}<\tfrac{1}{4}<\phantom{\tfrac{2}{7}}
<\tfrac{1}{3}<\phantom{\tfrac{3}{8}}<\tfrac{2}{5}<\tfrac{3}{7}<
\tfrac{1}{2}<\tfrac{4}{7}<\tfrac{3}{5}<\phantom{\tfrac{5}{8}}<\tfrac{2}{3}<\phantom{\tfrac{5}{7}}<\tfrac{3}{4}<\tfrac{4}{5}
<\phantom{\tfrac{5}{6}}<\phantom{\tfrac{6}{7}}<
\phantom{\tfrac{7}{8}}<\tfrac{1}{1}\bigr)\; ,\\
\mathcal{F}_{2^{s+1}m} &= \bigl(\tfrac{0}{1}<\phantom{\tfrac{1}{8}}<\phantom{\tfrac{1}{7}}<\phantom{\tfrac{1}{6}}<\phantom{\tfrac{1}{5}}
<\tfrac{1}{4}<\phantom{\tfrac{2}{7}}
<\tfrac{1}{3}<\phantom{\tfrac{3}{8}}<\phantom{\tfrac{2}{5}}<\phantom{\tfrac{3}{7}}<
\tfrac{1}{2}<\phantom{\tfrac{4}{7}}<\phantom{\tfrac{3}{5}}<\phantom{\tfrac{5}{8}}<\tfrac{2}{3}<\phantom{\tfrac{5}{7}}
<\tfrac{3}{4}<\phantom{\tfrac{4}{5}}<\phantom{\tfrac{5}{6}}<\phantom{\tfrac{6}{7}}<
\phantom{\tfrac{7}{8}}<\tfrac{1}{1}\bigr)\; ,\\
\mathcal{F}\bigl(\mathbb{B}(2^{s+1}m),2^s m\bigr) &= \bigl(\tfrac{0}{1}<\phantom{\tfrac{1}{8}}<\phantom{\tfrac{1}{7}}<\phantom{\tfrac{1}{6}}<\phantom{\tfrac{1}{5}}
<\phantom{\tfrac{1}{4}}<\phantom{\tfrac{2}{7}}
<\tfrac{1}{3}<\phantom{\tfrac{3}{8}}<\phantom{\tfrac{2}{5}}<\phantom{\tfrac{3}{7}}<
\tfrac{1}{2}<\phantom{\tfrac{4}{7}}<\phantom{\tfrac{3}{5}}<\phantom{\tfrac{5}{8}}<\tfrac{2}{3}<\phantom{\tfrac{5}{7}}
<\phantom{\tfrac{3}{4}}<\phantom{\tfrac{4}{5}}<\phantom{\tfrac{5}{6}}<\phantom{\tfrac{6}{7}}<
\phantom{\tfrac{7}{8}}<\tfrac{1}{1}\bigr)\; ,\\
\mathcal{F}_{2^s m}&\phantom{=}\\
=\mathcal{F}\bigl(\mathbb{B}(2^s m),2^{s-1}m\bigr) &= \bigl(\tfrac{0}{1}<\phantom{\tfrac{1}{8}}<\phantom{\tfrac{1}{7}}<\phantom{\tfrac{1}{6}}<\phantom{\tfrac{1}{5}}
<\phantom{\tfrac{1}{4}}<\phantom{\tfrac{2}{7}}
<\phantom{\tfrac{1}{3}}<\phantom{\tfrac{3}{8}}<\phantom{\tfrac{2}{5}}<\phantom{\tfrac{3}{7}}<
\tfrac{1}{2}<\phantom{\tfrac{4}{7}}<\phantom{\tfrac{3}{5}}<\phantom{\tfrac{5}{8}}<\phantom{\tfrac{2}{3}}<\phantom{\tfrac{5}{7}}
<\phantom{\tfrac{3}{4}}<\phantom{\tfrac{4}{5}}<\phantom{\tfrac{5}{6}}<\phantom{\tfrac{6}{7}}<
\phantom{\tfrac{7}{8}}<\tfrac{1}{1}\bigr)\; ,\\
\mathcal{F}_{2^{s-1}m}&= \bigl(\tfrac{0}{1}<\phantom{\tfrac{1}{8}}<\phantom{\tfrac{1}{7}}<\phantom{\tfrac{1}{6}}<\phantom{\tfrac{1}{5}}
<\phantom{\tfrac{1}{4}}<\phantom{\tfrac{2}{7}}
<\phantom{\tfrac{1}{3}}<\phantom{\tfrac{3}{8}}<\phantom{\tfrac{2}{5}}<\phantom{\tfrac{3}{7}}<
\phantom{\tfrac{1}{2}}<\phantom{\tfrac{4}{7}}<\phantom{\tfrac{3}{5}}<\phantom{\tfrac{5}{8}}<\phantom{\tfrac{2}{3}}<\phantom{\tfrac{5}{7}}
<\phantom{\tfrac{3}{4}}<\phantom{\tfrac{4}{5}}<\phantom{\tfrac{5}{6}}<\phantom{\tfrac{6}{7}}<
\phantom{\tfrac{7}{8}}<\tfrac{1}{1}\bigr)\; .
\end{align*}
}
It follows from~Eq.~{\rm(\ref{eq:35})} that the map
{\small
\begin{align*}
\mathcal{F}_{2^{s+2}m}\supset\left[\begin{smallmatrix}1&0\\2&1\end{smallmatrix}\right]\cdot\mathcal{F}_{2^s m}
=\left[\begin{smallmatrix}1&0\\2&1\end{smallmatrix}\right]\cdot
\Bigl(\left[\begin{smallmatrix}\!0\!\\ \!1\!\end{smallmatrix}\right],\left[\begin{smallmatrix}\!1\!\\ \!2\!\end{smallmatrix}\right],\left[\begin{smallmatrix}\!1\!\\ \!1\!\end{smallmatrix}\right]\Bigr)
&=\Bigl(\left[\begin{smallmatrix}\!0\!\\ \!1\!\end{smallmatrix}\right],\left[\begin{smallmatrix}\!1\!\\ \!4\!\end{smallmatrix}\right],\left[\begin{smallmatrix}\!1\!\\ \!3\!\end{smallmatrix}\right]\Bigr)\\
\to
\left[\begin{smallmatrix}0&1\\-1&3\end{smallmatrix}\right]\cdot\mathcal{F}_{2^s m}=
\left[\begin{smallmatrix}0&1\\-1&3\end{smallmatrix}\right]\cdot
\Bigl(\left[\begin{smallmatrix}\!0\!\\ \!1\!\end{smallmatrix}\right],\left[\begin{smallmatrix}\!1\!\\ \!2\!\end{smallmatrix}\right],\left[\begin{smallmatrix}\!1\!\\ \!1\!\end{smallmatrix}\right]\Bigr)
&=\Bigl(\left[\begin{smallmatrix}\!1\!\\ \!3\!\end{smallmatrix}\right],\left[\begin{smallmatrix}\!2\!\\ \!5\!\end{smallmatrix}\right],\left[\begin{smallmatrix}\!1\!\\ \!2\!\end{smallmatrix}\right]\Bigr)\subset\mathcal{F}_{2^{s+2}m}
\; , \ \ \
\left[\begin{smallmatrix}\!h\!\\ \!k\!\end{smallmatrix}\right]\mapsto
\left[\begin{smallmatrix}-2&1\\-7&3\end{smallmatrix}\right]\cdot
\left[\begin{smallmatrix}\!h\!\\ \!k\!\end{smallmatrix}\right]\; ,
\end{align*}
}
is order--preserving and bijective. Eqs.~{\rm(\ref{eq:47})} and~{\rm(\ref{eq:60})} imply that the maps
{\small
\begin{align*}
\mathcal{F}_{2^{s+2}m}\supset\Bigl(\left[\begin{smallmatrix}\!0\!\\ \!1\!\end{smallmatrix}\right],\left[\begin{smallmatrix}\!1\!\\ \!4\!\end{smallmatrix}\right],\left[\begin{smallmatrix}\!1\!\\ \!3\!\end{smallmatrix}\right]\Bigr)
&\to
\Bigl(\left[\begin{smallmatrix}\!1\!\\ \!3\!\end{smallmatrix}\right],\left[\begin{smallmatrix}\!2\!\\ \!5\!\end{smallmatrix}\right],\left[\begin{smallmatrix}\!1\!\\ \!2\!\end{smallmatrix}\right]\Bigr)\subset\mathcal{F}_{2^{s+2}m}
\; , &
\left[\begin{smallmatrix}\!h\!\\ \!k\!\end{smallmatrix}\right]&\mapsto
\left[\begin{smallmatrix}-2&1\\-3&2\end{smallmatrix}\right]\cdot
\left[\begin{smallmatrix}\!h\!\\ \!k\!\end{smallmatrix}\right]\; ,\\
\intertext
{\normalsize and}
\mathcal{F}_{2^{s+2}m}\supset\Bigl(\left[\begin{smallmatrix}\!1\!\\ \!3\!\end{smallmatrix}\right],\left[\begin{smallmatrix}\!2\!\\ \!5\!\end{smallmatrix}\right],\left[\begin{smallmatrix}\!1\!\\ \!2\!\end{smallmatrix}\right]\Bigr)
&\to
\Bigl(\left[\begin{smallmatrix}\!1\!\\ \!3\!\end{smallmatrix}\right],\left[\begin{smallmatrix}\!2\!\\ \!5\!\end{smallmatrix}\right],\left[\begin{smallmatrix}\!1\!\\ \!2\!\end{smallmatrix}\right]\Bigr)
\; , &
\left[\begin{smallmatrix}\!h\!\\ \!k\!\end{smallmatrix}\right]&\mapsto\!\!
\underset{\overset{\uparrow}{{\text{\rm involutory}}}}{\left[\begin{smallmatrix}1&0\\5&-1\end{smallmatrix}\right]}\!\!
\cdot
\left[\begin{smallmatrix}\!h\!\\ \!k\!\end{smallmatrix}\right]\; ,\ \ \ \ \
\left[\begin{smallmatrix}\!2\!\\ \!5\!\end{smallmatrix}\right]\mapsto\left[\begin{smallmatrix}\!2\!\\ \!5\!\end{smallmatrix}\right]\; ,
\end{align*}
}
are order--reversing and bijective.
\end{example}

\section{Farey (Sub)sequences and Monotone Maps}

\label{comp-maps}

The following statement, that uses the approach sketched in~Remark~\ref{th:4}(i), is based on the properties of maps~(\ref{eq:4}) and~(\ref{eq:15}) recalled in~Section~\ref{section2s}.

\begin{lemma}
\label{th:3}
Let $m$ and $n$ be positive integers, $n\geq 2m$. Suppose $s:=\lfloor\log_2 (n/m)\rfloor$.

For any ordered collection
\begin{equation}
(\boldsymbol{M}_1,\ldots,\boldsymbol{M}_{s})\in\left\{\left[\begin{smallmatrix}1&0\\ 1&1\end{smallmatrix}\right],\left[\begin{smallmatrix}0&1\\-1&2\end{smallmatrix}\right]\right\}^{s}\; ,
\end{equation}
of length $s$, whose entries are elements $\left[\begin{smallmatrix}1&0\\ 1&1\end{smallmatrix}\right]$ or $\left[\begin{smallmatrix}0&1\\-1&2\end{smallmatrix}\right]$ of\/ $\SL2Z$, the map
\begin{equation}
\label{eq:82}
\begin{split}
\mathcal{F}_{m}
&\to\mathcal{F}_{n}\; , \\
\left[\begin{smallmatrix}\!h\!\\ \!k\!\end{smallmatrix}\right]&\mapsto
\prod_{i=1}^{s}\boldsymbol{M}_{s-i+1}
\cdot
\left[\begin{smallmatrix}\!h\!\\ \!k\!\end{smallmatrix}\right]\ \ \
\begin{cases}
\leq\frac{1}{2}\; ,&\text{if\/ $\boldsymbol{M}_{s}=\left[\begin{smallmatrix}1&0\\ 1&1\end{smallmatrix}\right]$}\; ,\\
\geq\frac{1}{2}\; ,&\text{if\/ $\boldsymbol{M}_{s}=\left[\begin{smallmatrix}0&1\\-1&2\end{smallmatrix}\right]$}\; ,
\end{cases}
\end{split}
\end{equation}
is order--preserving and injective.
\end{lemma}

Since the matrix product from~Eq.~(\ref{eq:82}) is an element of~$\SL2Z$, the image of any Farey sequence retains the two generic properties of such sequences:

\begin{remark}
If\hfill $\tfrac{h_j}{k_j}<\tfrac{h_{j+1}}{k_{j+1}}<\tfrac{h_{j+2}}{k_{j+2}}$\hfill are\hfill three\hfill consecutive\hfill fractions\hfill in\hfill the\hfill subse-\newline quence~$\prod_{i=1}^{s}\boldsymbol{M}_{s-i+1}\cdot\mathcal{F}_{m}$ of the Farey sequence $\mathcal{F}_{n}$, considered in Lemma~{\rm{\ref{th:3}}}, then~$k_j h_{j+1}$ $-h_j k_{j+1}=1$, and $\tfrac{h_{j+1}}{k_{j+1}}=
\tfrac{h_j+h_{j+2}}{\gcd(h_j+h_{j+2},k_j+k_{j+2})}\!\!\Bigm/\!\!\tfrac{k_j+k_{j+2}}{\gcd(h_j+h_{j+2},k_j+k_{j+2})}$.
\end{remark}

We\hfill conclude\hfill this\hfill section\hfill with\hfill a\hfill straightforward\hfill extension\hfill of\hfill Proposition~\ref{th:5},\hfill see\newline also~Example~\ref{th:6}.

\begin{theorem}
\label{th:7}
Let $m$ and $n$ be positive integers, $n\geq 2m$. Suppose $s:=\lfloor\log_2 (n/m)\rfloor$. Given two ordered collections of matrices
\begin{equation}
(\boldsymbol{M}_1,\ldots,\boldsymbol{M}_{s}),\; (\boldsymbol{N}_1,\ldots,\boldsymbol{N}_{s})\in\left\{\left[\begin{smallmatrix}1&0\\ 1&1\end{smallmatrix}\right],\left[\begin{smallmatrix}0&1\\-1&2\end{smallmatrix}\right]\right\}^{s}\; ,
\end{equation}
suppose
\begin{equation}
\mathbf{M}:=\prod_{i=1}^{s}\boldsymbol{M}_{s-i+1}\; ,\ \ \ \mathbf{N}:=\prod_{i=1}^{s}\boldsymbol{N}_{s-i+1}\; .
\end{equation}

The map
\begin{equation}
\begin{split}
\mathcal{F}_{n}\supset\mathbf{M}\cdot\mathcal{F}_{m}&\to
\mathbf{N}\cdot\mathcal{F}_{m}\subset\mathcal{F}_{n}\; ,\\
\left[\begin{smallmatrix}\!h\!\\ \!k\!\end{smallmatrix}\right]&\mapsto
\mathbf{N}\cdot
\mathbf{M}^{-1}
\cdot\left[\begin{smallmatrix}\!h\!\\ \!k\!\end{smallmatrix}\right]\; ,
\end{split}
\end{equation}
between subsequences\/ $\mathbf{M}\cdot\mathcal{F}_{m}$ and\/ $\mathbf{N}\cdot\mathcal{F}_{m}$ of the Farey sequence\/ $\mathcal{F}_{n}$, is order--preserving and bijective.

The map
\begin{equation}
\begin{split}
\mathcal{F}_{n}\supset\mathbf{M}\cdot\mathcal{F}_{m}&\to
\mathbf{N}\cdot\mathcal{F}_{m}\subset\mathcal{F}_{n}\; ,\\
\left[\begin{smallmatrix}\!h\!\\ \!k\!\end{smallmatrix}\right]&\mapsto
\mathbf{N}\cdot
\left[\begin{smallmatrix}-1&1\\ 0&1\end{smallmatrix}\right]\cdot
\mathbf{M}^{-1}
\cdot\left[\begin{smallmatrix}\!h\!\\ \!k\!\end{smallmatrix}\right]\; ,
\end{split}
\end{equation}
is order--reversing and bijective; in particular, the map
\begin{equation}
\begin{split}
\mathcal{F}_{n}\supset\mathbf{M}\cdot\mathcal{F}_{m}&\to
\mathbf{M}\cdot\mathcal{F}_{m}\; ,\\
\left[\begin{smallmatrix}\!h\!\\ \!k\!\end{smallmatrix}\right]&\mapsto
\underbrace{\mathbf{M}\cdot
\left[\begin{smallmatrix}-1&1\\ 0&1\end{smallmatrix}\right]\cdot
\mathbf{M}^{-1}}_{\text{\rm involutory}}
\cdot\left[\begin{smallmatrix}\!h\!\\ \!k\!\end{smallmatrix}\right]\; ,
\end{split}
\end{equation}
with the fixed point\/ $\mathbf{M}\cdot\left[\begin{smallmatrix}\!1\!\\ \!2\!\end{smallmatrix}\right]$ in the case of $m>1$, is order--reversing and bijective.
\end{theorem}

\section{Appendix: Miscellany}

\label{sec-app}

\subsection{Farey (Sub)sequences: List of Notation}

{\small
\renewcommand{\arraystretch}{1.4}
\begin{center}
\begin{tabular}{|c|c|}
\hline
\it Sequence & \it Definition / Description\\
\hline
\hline
$\mathcal{F}_n$ & $\{\tfrac{0}{1}\}\dot{\cup}\!\left(\tfrac{h}{k}\in\mathbb{Q}:\ 1\leq h< k\leq n,\ \gcd(k,h)=1\right)\!\dot{\cup}\{\tfrac{1}{1}\}$\\
\hline
$\mathcal{F}_n^m$,\ $m\geq 1$ & $\left(\tfrac{h}{k}\in\mathcal{F}_n:\ h\leq m\right)$\\
\hline
$\mathcal{G}_n^m$,\ $m\leq n-1$ & $\left(\tfrac{h}{k}\in\mathcal{F}_n:\ m+k-n\leq h\right)$\\
\hline\hline
$\mathcal{F}(\mathbb{B}(n),m)$ & $\mathcal{F}_n^m\cap\mathcal{G}_n^m=\left(\tfrac{h}{k}\in\mathcal{F}_n:\ m+k-n\leq h\leq m\right)$\\
\hline
$\mathcal{F}\bigl(\mathbb{B}(2m),m\bigr)$ &
$\left(\tfrac{h}{k}\in\mathcal{F}_{2m}:\ k-m\leq h\leq m\right)$\\
\hline
$\mathcal{F}^{\leq\frac{1}{2}}(\mathbb{B}(n),m)$ & $\left(\tfrac{h}{k}\in\mathcal{F}(\mathbb{B}(n),m):\ \tfrac{h}{k}\leq\tfrac{1}{2}\right)$\\
\hline
$\mathcal{F}^{\geq\frac{1}{2}}(\mathbb{B}(n),m)$ & $\left(\tfrac{h}{k}\in\mathcal{F}(\mathbb{B}(n),m):\ \tfrac{h}{k}\geq\tfrac{1}{2}\right)$\\
\hline
\hline
$\mathcal{F}\bigl(\mathbb{B}(n),m\bigr)^{\ell}$, $1\leq\ell\leq m$ & $\left(\tfrac{h}{k}\in\mathcal{F}\bigl(\mathbb{B}(n),m\bigr):\ h\leq \ell\right)=\mathcal{F}\bigl(\mathbb{B}(n-m+\ell),\ell\bigr)$\\
\hline
$\mathcal{F}\bigl(\mathbb{B}(2m),m\bigr)^{\ell}$
&
$\mathcal{F}\bigl(\mathbb{B}(m+\ell),\ell\bigr)$\\
\hline\hline
$\mathcal{G}\bigl(\mathbb{B}(n),m\bigr)^{\ell}$, $m\leq\ell\leq n-1$ &
$\left(\tfrac{h}{k}\in\mathcal{F}\bigl(\mathbb{B}(n),m\bigr):\ \ell+k-n\leq h\right)=\mathcal{F}\bigl(\mathbb{B}(n+m-\ell),m\bigr)$\\
\hline
$\mathcal{G}\bigl(\mathbb{B}(2m),m\bigr)^{\ell}$
&
$\mathcal{F}\bigl(\mathbb{B}(3m-\ell),m\bigr)$\\
\hline
\end{tabular}
\end{center}
\renewcommand{\arraystretch}{1.0}
}

\subsection{The Cardinalities of Subsequences}

Let us calculate the number of fractions in several subsequences of the~Farey sequences; $\bar\mu(\cdot)$ denotes the number--theoretic M\"{o}bius function.

We have
\begin{equation}
\label{eq:67}
\begin{split}
|\mathcal{F}_n^m|-\bigl|\mathcal{F}\bigl(\mathbb{B}(n),m\bigr)\bigr|&=
\underbrace{\tfrac{3}{2}+\sum_{d=1}^m\bar\mu(d)\left\lfloor\!\tfrac{m}{d}\!\right\rfloor\left(\left\lfloor\!\tfrac{n}{d}\!\right\rfloor
-\tfrac{1}{2}\left\lfloor\!\tfrac{m}{d}\!\right\rfloor\right)}_{|\mathcal{F}_n^m|}-
\biggl(2+
\sum_{d=1}^m\bar\mu(d)\left\lfloor\!\tfrac{m}{d}\!\right\rfloor\left\lfloor\!\tfrac{n-m}{d}\!\right\rfloor\biggr)
\\&=-\tfrac{1}{2}+\sum_{d=1}^m\bar\mu(d)\left\lfloor\!\tfrac{m}{d}\!\right\rfloor\left(\left\lfloor\!\tfrac{n}{d}\right\rfloor
-\tfrac{1}{2}\left\lfloor\!\tfrac{m}{d}\!\right\rfloor-\left\lfloor\!\tfrac{n-m}{d}\!\right\rfloor\right)\; ,
\end{split}
\end{equation}
and
\begin{equation}
\begin{split}
|\mathcal{G}_n^m|-\bigl|\mathcal{F}\bigl(\mathbb{B}(n),m\bigr)\bigr|&=
\underbrace{\tfrac{3}{2}+\sum_{d=1}^{n-m}\bar\mu(d)\left\lfloor\!\tfrac{n-m}{d}\!\right\rfloor\left(\left\lfloor\!\tfrac{n}{d}\!\right\rfloor
-\tfrac{1}{2}\left\lfloor\!\tfrac{n-m}{d}\!\right\rfloor\right)}_{|\mathcal{G}_n^m|}-
\biggl(2+\sum_{d=1}^{n-m}
\bar\mu(d)\left\lfloor\!\tfrac{m}{d}\!\right\rfloor\left\lfloor\!\tfrac{n-m}{d}\!\right\rfloor\biggr)
\\&=-\tfrac{1}{2}+\sum_{d=1}^{n-m}\bar\mu(d)\left\lfloor\!\tfrac{n-m}{d}\!\right\rfloor\left(\left\lfloor\!\tfrac{n}{d}\right\rfloor
-\tfrac{1}{2}\left\lfloor\!\tfrac{n-m}{d}\!\right\rfloor-\left\lfloor\!\tfrac{m}{d}\!\right\rfloor\right)\; ;
\end{split}
\end{equation}
in particular,
\begin{equation}
\begin{split}
|\mathcal{F}_{2m}^m|-\bigl|\mathcal{F}\bigl(\mathbb{B}(2m),m\bigr)\bigr|&=|\mathcal{G}_{2m}^m|-\bigl|\mathcal{F}\bigl(\mathbb{B}(2m),m\bigr)\bigr|\\&=
-\tfrac{1}{2}+\sum_{d=1}^m\bar\mu(d)\left\lfloor\!\tfrac{m}{d}\!\right\rfloor\left(\left\lfloor\!\tfrac{2m}{d}\right\rfloor
-\tfrac{3}{2}\left\lfloor\!\tfrac{m}{d}\!\right\rfloor\right)\; .
\end{split}
\end{equation}
As a consequence,
\begin{equation}
\label{eq:10}
\begin{split}
|\mathcal{F}_n|-\bigl|\mathcal{F}\bigl(\mathbb{B}(n),m\bigr)\bigr|&=
\left(|\mathcal{F}_n^m|-\bigl|\mathcal{F}\bigl(\mathbb{B}(n),m\bigr)\bigr|\right)+
\left(|\mathcal{G}_n^m|-\bigl|\mathcal{F}\bigl(\mathbb{B}(n),m\bigr)\bigr|\right)\\
&=-4+\sum_{d\geq 1}\bar\mu(d)\left\lfloor\!\tfrac{n}{d}\!\right\rfloor
\left(\left\lfloor\!\tfrac{m}{d}\!\right\rfloor+\left\lfloor\!\tfrac{n-m}{d}\!\right\rfloor\right)
\; ,
\end{split}
\end{equation}
and
\begin{equation}
\begin{split}
|\mathcal{F}_{2m}|-\bigl|\mathcal{F}\bigl(\mathbb{B}(2m),m\bigr)\bigr|&=
2\left(|\mathcal{F}_{2m}^m|-\bigl|\mathcal{F}\bigl(\mathbb{B}(2m),m\bigr)\bigr|\right)\\&=
-4+2\sum_{d=1}^m\bar\mu(d)\left\lfloor\!\tfrac{2m}{d}\!\right\rfloor\left\lfloor\!\tfrac{m}{d}\!\right\rfloor\; .
\end{split}
\end{equation}

{\small
\renewcommand{\arraystretch}{1.4}
\begin{center}
\begin{tabular}{|c|c|}
\hline
\it Quantity & \it Formula\\
\hline
\hline
$|\mathcal{F}_n|$
& $\frac{3}{2}+\frac{1}{2}\sum_{d=1}^n\bar\mu(d)\left\lfloor\!\frac{n}{d}\!\right\rfloor^2$
\\
\hline\hline
$|\mathcal{F}_n^m|=|\mathcal{G}_n^{n-m}|$ & $\tfrac{3}{2}+\sum_{d\geq 1}\bar\mu(d)\left\lfloor\!\frac{m}{d}\!\right\rfloor\left(\left\lfloor\!\frac{n}{d}\!\right\rfloor
-\tfrac{1}{2}\left\lfloor\!\frac{m}{d}\!\right\rfloor\right)$\\
\hline
$|\mathcal{F}_{2m}^m|=|\mathcal{G}_{2m}^m|$ &
$\tfrac{3}{2}+\sum_{d=1}^m\bar\mu(d)\left\lfloor\!\tfrac{m}{d}\!\right\rfloor\left(\left\lfloor\!\tfrac{2m}{d}\!\right\rfloor
-\tfrac{1}{2}\left\lfloor\!\tfrac{m}{d}\!\right\rfloor\right)$\\
\hline\hline
$|\mathcal{F}(\mathbb{B}(n),m)|$ & $2+\sum_{d\geq 1}\bar\mu(d)\left\lfloor\!\frac{m}{d}\!\right\rfloor\left\lfloor\!\frac{n-m}{d}\!\right\rfloor$\\
\hline
$|\mathcal{F}(\mathbb{B}(2m),m)|$ & $2+\sum_{d=1}^m\bar\mu(d)\left\lfloor\!\frac{m}{d}\!\right\rfloor^2$\\
\hline
\hline
$|\mathcal{F}_n^m|-\bigl|\mathcal{F}\bigl(\mathbb{B}(n),m\bigr)\bigr|$ &
$-\tfrac{1}{2}+\sum_{d=1}^m\bar\mu(d)\left\lfloor\!\tfrac{m}{d}\!\right\rfloor\left(\left\lfloor\!\tfrac{n}{d}\right\rfloor
-\tfrac{1}{2}\left\lfloor\!\tfrac{m}{d}\!\right\rfloor-\left\lfloor\!\tfrac{n-m}{d}\!\right\rfloor\right)$
\\
\hline
$|\mathcal{G}_n^m|-\bigl|\mathcal{F}\bigl(\mathbb{B}(n),m\bigr)\bigr|$ &
$-\tfrac{1}{2}+\sum_{d=1}^m\bar\mu(d)\left\lfloor\!\tfrac{n-m}{d}\!\right\rfloor\left(\left\lfloor\!\tfrac{n}{d}\right\rfloor
-\tfrac{1}{2}\left\lfloor\!\tfrac{n-m}{d}\!\right\rfloor-\left\lfloor\!\tfrac{m}{d}\!\right\rfloor\right)$
\\
\hline
$|\mathcal{F}_{2m}^m|-\bigl|\mathcal{F}\bigl(\mathbb{B}(2m),m\bigr)\bigr|=|\mathcal{G}_{2m}^m|-\bigl|\mathcal{F}\bigl(\mathbb{B}(2m),m\bigr)\bigr|$ &
$-\tfrac{1}{2}+\sum_{d=1}^m\bar\mu(d)\left\lfloor\!\tfrac{m}{d}\!\right\rfloor\left(\left\lfloor\!\tfrac{2m}{d}\right\rfloor
-\tfrac{3}{2}\left\lfloor\!\tfrac{m}{d}\!\right\rfloor\right)$\\
\hline
\hline
$|\mathcal{F}_n|-\bigl|\mathcal{F}\bigl(\mathbb{B}(n),m\bigr)\bigr|$ &
$-4+\sum_{d\geq 1}\bar\mu(d)\left\lfloor\!\tfrac{n}{d}\!\right\rfloor
\left(\left\lfloor\!\tfrac{m}{d}\!\right\rfloor+\left\lfloor\!\tfrac{n-m}{d}\!\right\rfloor\right)$\\
\hline
$|\mathcal{F}_{2m}|-\bigl|\mathcal{F}\bigl(\mathbb{B}(2m),m\bigr)\bigr|$ &
$-4+2\sum_{d=1}^m\bar\mu(d)\left\lfloor\!\tfrac{2m}{d}\!\right\rfloor\left\lfloor\!\tfrac{m}{d}\!\right\rfloor$\\
\hline
\end{tabular}
\end{center}
\renewcommand{\arraystretch}{1.0}
}
The recursive expression for the number of fractions in the Farey sequence---the basis of similar recursive expressions for the other quantities, which is the inverse of the formula~\mbox{$|\mathcal{F}_n|=\frac{3}{2}+\frac{1}{2}\sum_{d=1}^n\bar\mu(d)\left\lfloor\!\frac{n}{d}\!\right\rfloor^2$}, is as follows:
\begin{equation}
|\mathcal{F}_n|=\frac{1}{2}(n+3)n-\sum_{d=2}^n|\mathcal{F}_{\lfloor n/d\rfloor}|\ ,
\end{equation}
see~\cite[\href{http://oeis.org/search?q=A005728&language=english&go=Search}{A005728}]{OEIS}.

\subsection{Matrix Products}

Let us survey several subproducts of the matrix products mentioned in Remark~\ref{th:4}, Lemma~\ref{th:3}, Theorem~\ref{th:7}, and in the proof of Proposition~\ref{th:5}.

If $j\in\mathbb{Z}$ then
\begin{equation}
\left[\begin{smallmatrix}1&0\\ 1&1\end{smallmatrix}\right]^j=\left[\begin{smallmatrix}1&0\\j&1\end{smallmatrix}\right]\; ,\ \ \
\left[\begin{smallmatrix}0&1\\-1&2\end{smallmatrix}\right]^j=\left[\begin{smallmatrix}1-j&j\\-j&1+j\end{smallmatrix}\right]\; ;\ \ \
\left[\begin{smallmatrix}1&0\\ 1&1\end{smallmatrix}\right]^{-1}=\left[\begin{smallmatrix}1&0\\-1&1\end{smallmatrix}\right]\; ,\ \ \
\left[\begin{smallmatrix}0&1\\-1&2\end{smallmatrix}\right]^{-1}=\left[\begin{smallmatrix}2&-1\\1&0\end{smallmatrix}\right]\; .
\end{equation}

If $i\in\mathbb{Z}$ then
\begin{equation}
\left[\begin{smallmatrix}1&0\\ 1&1\end{smallmatrix}\right]^i\cdot
\left[\begin{smallmatrix}0&1\\-1&2\end{smallmatrix}\right]^j=\left[\begin{smallmatrix}1-j&j\\i-j-ij&1+j+ij\end{smallmatrix}\right]\; ,\ \ \
\left[\begin{smallmatrix}0&1\\-1&2\end{smallmatrix}\right]^i\cdot
\left[\begin{smallmatrix}1&0\\ 1&1\end{smallmatrix}\right]^j=\left[\begin{smallmatrix}1-i+ij&i\\-i+j+ij&1+i\end{smallmatrix}\right]\; ;
\end{equation}
we also have
\begin{align}
\left[\begin{smallmatrix}1&0\\ 1&1\end{smallmatrix}\right]^i\cdot
\left[\begin{smallmatrix}-1&1\\0&1\end{smallmatrix}\right]&=\left[\begin{smallmatrix}-1&1\\-i&1+i\end{smallmatrix}\right]\; , &
\left[\begin{smallmatrix}0&1\\-1&2\end{smallmatrix}\right]^i\cdot
\left[\begin{smallmatrix}-1&1\\ 0&1\end{smallmatrix}\right]&=\left[\begin{smallmatrix}-1+i&1\\i&1\end{smallmatrix}\right]\; ,\\
\left[\begin{smallmatrix}-1&1\\0&1\end{smallmatrix}\right]\cdot\left[\begin{smallmatrix}1&0\\ 1&1\end{smallmatrix}\right]^j\cdot
&=\left[\begin{smallmatrix}-1+j&1\\j&1\end{smallmatrix}\right]\; , & \left[\begin{smallmatrix}-1&1\\ 0&1\end{smallmatrix}\right]\cdot\left[\begin{smallmatrix}0&1\\-1&2\end{smallmatrix}\right]^j
&=\left[\begin{smallmatrix}-1&1\\-j&1+j\end{smallmatrix}\right]\; ,
\end{align}
and
\begin{align}
\left[\begin{smallmatrix}1&0\\ 1&1\end{smallmatrix}\right]^i\cdot
\left[\begin{smallmatrix}-1&1\\0&1\end{smallmatrix}\right]\cdot\left[\begin{smallmatrix}0&1\\-1&2\end{smallmatrix}\right]^j
&=\left[\begin{smallmatrix}-1&1\\-i-j&1+i+j\end{smallmatrix}\right]\; ,\\
\left[\begin{smallmatrix}0&1\\-1&2\end{smallmatrix}\right]^i\cdot
\left[\begin{smallmatrix}-1&1\\0&1\end{smallmatrix}\right]\cdot
\left[\begin{smallmatrix}1&0\\ 1&1\end{smallmatrix}\right]^j\cdot
&=\left[\begin{smallmatrix}-1+i+j&1\\i+j&1\end{smallmatrix}\right]\; ,\\
\left[\begin{smallmatrix}1&0\\ 1&1\end{smallmatrix}\right]^i\cdot
\left[\begin{smallmatrix}-1&1\\0&1\end{smallmatrix}\right]\cdot\left[\begin{smallmatrix}1&0\\ 1&1\end{smallmatrix}\right]^j
&=\left[\begin{smallmatrix}-1+j&1\\-i+j+ij&1+i\end{smallmatrix}\right]\; ,\\
\left[\begin{smallmatrix}0&1\\-1&2\end{smallmatrix}\right]^i\cdot
\left[\begin{smallmatrix}-1&1\\0&1\end{smallmatrix}\right]\cdot
\left[\begin{smallmatrix}0&1\\-1&2\end{smallmatrix}\right]^j
&=\left[\begin{smallmatrix}-1+i-ij&1+ij\\i-j-ij&1+j+ij\end{smallmatrix}\right]\; .
\end{align}

\subsection{Order--reversing and bijective mapping\/ $\frac{h}{k}\mapsto\frac{k-h}{k}$}

In the context of Farey sequences, the mapping $\left[\begin{smallmatrix}\!h\!\\
\!k\!\end{smallmatrix}\right]\mapsto\left[\begin{smallmatrix}-1&1\\0&1\end{smallmatrix}\right]\cdot
\left[\begin{smallmatrix}\!h\!\\ \!k\!\end{smallmatrix}\right]$ determined by the involutory matrix~$\left[\begin{smallmatrix}-1&1\\0&1\end{smallmatrix}\right]$ is ubiquitous:
{\small
\renewcommand{\arraystretch}{1.4}
\begin{center}
\begin{tabular}{|c|rcl|}
\hline
\multicolumn{4}{|c|}{\it Order--reversing and bijective mapping\/ $\frac{h}{k}\mapsto\frac{k-h}{k}$}\\
\hline
\hline
 & $\mathcal{F}_n$ & $\to$ & $\mathcal{F}_n$\\
 & $\mathcal{F}_n^m$ & $\to$ & $\mathcal{G}_n^{n-m}$\\
 & $\mathcal{G}_n^m$ & $\to$ & $\mathcal{F}_n^{n-m}$\\
 & $\mathcal{F}_{2m}^m$ & $\to$ & $\mathcal{G}_{2m}^m$\\
 & $\mathcal{G}_{2m}^m$ & $\to$ & $\mathcal{F}_{2m}^m$\\
$\left[\begin{smallmatrix}\!h\!\\
\!k\!\end{smallmatrix}\right]\mapsto
\left[\begin{smallmatrix}-1&1\\0&1\end{smallmatrix}\right]\cdot
\left[\begin{smallmatrix}\!h\!\\ \!k\!\end{smallmatrix}\right]$
 & $\mathcal{F}(\mathbb{B}(n),m)$ & $\to$ & $\mathcal{F}(\mathbb{B}(n),n-m)$\\
 & $\mathcal{F}(\mathbb{B}(2m),m)$ & $\to$ & $\mathcal{F}(\mathbb{B}(2m),m)$\\
 & $\mathcal{F}(\mathbb{B}(n),m)^{\ell}$ & $\to$ & $\mathcal{F}(\mathbb{B}(n),n-\ell)^{n-m}$\\
 & $\mathcal{G}(\mathbb{B}(n),m)^{\ell}$ & $\to$ & $\mathcal{G}(\mathbb{B}(n),n-\ell)^{n-m}$\\
 & $\mathcal{F}(\mathbb{B}(2m),m)^{\ell}$ & $\to$ & $\mathcal{F}(\mathbb{B}(2m),2m-\ell)^m$\\
 & $\mathcal{G}(\mathbb{B}(2m),m)^{\ell}$ & $\to$ & $\mathcal{G}(\mathbb{B}(2m),2m-\ell)^m$\\
\hline
\end{tabular}
\end{center}
\renewcommand{\arraystretch}{1.0}
}

Let us assign to an ordered collection of matrices $(\boldsymbol{M}_1,\ldots,\boldsymbol{M}_t)\in\left\{\left[\begin{smallmatrix}1&0\\ 1&1\end{smallmatrix}\right],\left[\begin{smallmatrix}0&1\\-1&2\end{smallmatrix}\right]\right\}^t$ the collection
$(\boldsymbol{P}_1,\ldots,\boldsymbol{P}_t)$ such that
\begin{equation}
\boldsymbol{P}_i:=\begin{cases}\left[\begin{smallmatrix}0&1\\-1&2\end{smallmatrix}\right]
\; , &\text{if $\boldsymbol{M}_i=\left[\begin{smallmatrix}1&0\\ 1&1\end{smallmatrix}\right]$}\; ,\\
\left[\begin{smallmatrix}1&0\\ 1&1\end{smallmatrix}\right]\; ,
&\text{if $\boldsymbol{M}_i=\left[\begin{smallmatrix}0&1\\-1&2\end{smallmatrix}\right]$}\; ,
\end{cases}
\end{equation}
$1\leq i\leq t$. For any order $m$, the map
\begin{equation}
\begin{split}
\prod_{i=1}^t\boldsymbol{M}_{t-i+1}\cdot\mathcal{F}_m&\to\prod_{i=1}^t\boldsymbol{P}_{t-i+1}\cdot\mathcal{F}_m\; ,\\
\left[\begin{smallmatrix}\!h\!\\ \!k\!\end{smallmatrix}\right]&\mapsto
\left[\begin{smallmatrix}-1&1\\0&1\end{smallmatrix}\right]\cdot\left[\begin{smallmatrix}\!h\!\\ \!k\!\end{smallmatrix}\right]\; ,
\end{split}
\end{equation}
is order--reversing and bijective.

\end{document}